\documentclass[a4paper, 11pt]{article}
\usepackage[utf8]{inputenc}
\usepackage{hyperref}
\usepackage{amssymb,amsmath,amsfonts,mathdots,amsthm,textcomp,enumerate,mathrsfs,eufrak,version,mathtools} 
\usepackage{upgreek}
\usepackage{caption}
\usepackage{array}
\usepackage{lipsum}
\usepackage{multirow}
\usepackage{vmargin}
\usepackage{titlesec}
\usepackage{eurosym}
\usepackage[english]{babel}
\usepackage[nottoc,notlot,notlof]{tocbibind}
\usepackage[titletoc,toc,title]{appendix}
\titleformat{\subsubsection}{\Large\scshape\raggedright}{}{0em}{}[\titlerule]
\usepackage[final]{pdfpages}
\usepackage{colortbl}
\usepackage{url}
\usepackage[T1]{fontenc}
\usepackage{color}
\usepackage[all,cmtip]{xy}
\usepackage{tikz}
\usetikzlibrary{arrows.meta}
\tikzset{%
  >={Latex[width=2mm,length=2mm]},
            base/.style = {rectangle, rounded corners, draw=black,
                           minimum width=4cm, minimum height=1cm,
                           text centered, font=\sffamily},
}
\newtheorem{theo}{Theorem}[section]

\newtheorem{prop}[theo]{Proposition}
\newtheorem{lem}[theo]{Lemma}

\newcommand{\ch}[1]{{}^\vee#1}

\newcommand{\R}{\mathbb{R}}

\newcommand{\g}{\mathfrak g}

\newcommand{\C}{\mathbb{C}}

\title{Micro-packets containing generic representations}

\author{N. Arancibia Robert}

\begin{document}
\maketitle

\begin{abstract}
For a real group $G$, it is known from the work of Kostant and Vogan that the L-packet associated with an L-parameter $\varphi$ of $G$ contains a \emph{generic} representation if and only if the ${}^{\vee}G$-orbit in the variety of geometric parameters corresponding to $\varphi$ is open. In these notes, we generalize this result slightly by proving that the same equivalence holds when the L-packet of $\varphi$ is replaced by the micro-packet attached to $\varphi$ by Adams–Barbasch–Vogan. As a corollary, we deduce the Enhanced Shahidi Conjecture for real groups: the Arthur packet attached to an A-parameter $\psi$ of $G$ contains a generic representation if and only if $\psi|_{\mathrm{SL}_2}$ is trivial.
\end{abstract}

\tableofcontents

\section{Introduction}

Let $G$ be a connected reductive algebraic group defined over $\mathbb{R}$. 
In \cite{Langlands}, Langlands defined a partition of the set 
\(\Pi(G(\mathbb{R}))\) 
of (infinitesimal equivalence classes of) irreducible admissible representations of \(G(\mathbb{R})\) into \emph{L-packets}.  
These L-packets, denoted \(\Pi_{\varphi}(G(\mathbb{R}))\), are parameterized by 
L-parameters, i.e., 
${}^{\vee}G$-conjugacy classes of L-homomorphisms
\[
\varphi: W_{\mathbb{R}} \longrightarrow {}^\vee G^\Gamma,
\]
from the real Weil group \(W_{\mathbb{R}}\) of \(\Gamma = \mathrm{Gal}(\mathbb{C}/\mathbb{R})\) to the L-group 
${}^\vee G^\Gamma = {}^\vee G \rtimes \Gamma$, where ${}^\vee G$ is the Langlands dual group of \(G\) (\cite[Section~8]{Borel}).
We denote by $P\left({}^\vee G^\Gamma\right)$ the set of L-homomorphisms, by  
$\Phi\left({}^\vee G^\Gamma\right)$ the set ${}^{\vee}G$-conjugacy classes on $P\left({}^\vee G^\Gamma\right)$, and, as usual, we refer to the map
\begin{align}\label{eq:LLC}
\varphi\in \Phi\left({}^\vee G^\Gamma\right) \longmapsto \Pi_\varphi(G(\R)),
\end{align}
as the \emph{Local Langlands correspondence}.
Building on this framework, Adams, Barbasch and Vogan \cite{ABV} propose a reformulation of the Langlands correspondence using geometric tools.  
Their approach introduces a variety \(X({}^\vee G^\Gamma)\) that reparametrizes the set of L-homomorphisms.  
More precisely, \({}^\vee G\) acts on \(X({}^\vee G^\Gamma)\), and its orbits \(S \subset X({}^\vee G^\Gamma)\) are in bijection with the set \(\Phi({}^\vee G^\Gamma)\), providing a more geometric interpretation of the 
Local Langlands correspondence
\begin{align}\label{eq:ABVLLC}
\Pi_\varphi(G(\mathbb{R})) \longleftrightarrow \varphi \longleftrightarrow S_\varphi.
\end{align}
Moreover, in their quest to prove the Local Arthur conjectures \cite{Arthur84,Arthur89} 
and to define Arthur's conjectural A-packets for real groups, 
each ${}^\vee G$-orbit $S \subset X({}^\vee G^\Gamma)$ is attached to a packet 
$\Pi_S^{\mathrm{mic}}(G(\mathbb{R}))$ of irreducible admissible representations of $G(\mathbb{R})$. 
This packet, called a \emph{micro-packet} in \cite[Definition 19.15]{ABV}, 
is constructed using microlocal geometric methods and, among other properties, contains the L-packet 
\(\Pi_S(G(\mathbb{R}))\) corresponding to the \({}^\vee G\)-orbit \(S\) under \eqref{eq:ABVLLC}.

We now suppose that $G$ is quasisplit, and recall that an irreducible representation $\pi$ of $G(\R)$ is called \emph{generic} if it admits a Whittaker model (see for example \cite[Definition 3.1]{ABV}).  
It follows from \cite[Proposition 1.11]{ABV} together with the description of generic representations in \cite{Kostant78} and \cite{Vogan78} that an L-packet has a generic member if and only if the corresponding ${}^{\vee}G$-orbit $S \subset X({}^\vee G^\Gamma)$ is open.

The objective of these notes is to extend this result to the setting of micro-packets.  
More precisely, we prove in Theorem~\ref{theo:ESC} below that: \begin{quote}The micro-packet 
$\Pi_S^{\mathrm{mic}}(G(\mathbb{R}))$ contains a generic member 
if and only if the 
${}^{\vee}G$-orbit $S$ is open.\end{quote}  

Since, by  \cite[Section 4]{Arthur89}, to any A-parameter $\psi$ there corresponds an L-parameter 
$\varphi_\psi$, and therefore, under \eqref{eq:ABVLLC}, a ${}^{\vee}G$-orbit 
$S_\psi \subset X({}^\vee G^\Gamma)$, it follows that by showing $S_\psi$ is open if and only if 
$\psi|_{\mathrm{SL}_2}$ is trivial, 
we obtain, as a corollary, the Enhanced Shahidi Conjecture  \cite[Conjecture 1.6]{LiuShahidi}
for archimedean fields:  
\begin{quote}
For any quasi-split reductive group  $G$ defined over $\R$, the Arthur packet $\Pi_\psi(G(\R))$
has a generic member if and only if $\psi|_{\mathrm{SL}_2}$ is trivial.
\end{quote}

Let us briefly discuss the methods involved in the proof of Theorem~\ref{theo:ESC}.  
First, it suffices to establish the forward direction.  
Indeed, by \cite[Lemma~19.14(a)]{ABV}, for an open orbit \( S \), we have
\[
\Pi_S(G(\mathbb{R})) = \Pi_S^{\mathrm{mic}}(G(\mathbb{R})),
\]  
where \( \Pi_S(G(\mathbb{R})) \) denotes the \( L \)-packet of the \( L \)-parameter corresponding to \( S \) under \eqref{eq:ABVLLC}. 
As a consequence, a micro-packet containing a generic member reduces to the corresponding \( L \)-packet.  
Micro-packets are defined using the notion of characteristic cycle of a \( D \)-module.  
To prove the forward implication in Theorem~\ref{theo:ESC}, we explicitly compute these characteristic cycles in the case of \( D \)-modules associated to generic representations.
To do so, we exploit, in the case of representations with regular infinitesimal character, the existence of an action of the (integral) Weyl group on 
both the domain \cite[Definition 7.2.28]{greenbook} and codomain \cite{Hotta85}, \cite{Rossmann95} of the characteristic cycle map, which makes this map equivariant. 

The definition of a micro-packet is reviewed in Section~2, where 
we also describe the behavior of these packets under the translation principle.  
In Section~3, we study the action of the Weyl group that makes the characteristic cycle map equivariant. 
Building on the results of Section~3, we compute, in Theorem~\ref{theo:ccforlarge}, 
the characteristic cycle of the irreducible $D$-module corresponding to a generic 
representation under the local Langlands correspondence. 
As a corollary, we obtain Theorem~\ref{theo:ESC}, together with the Enhanced Shahidi Conjecture.

For \( p \)-adic groups, the Enhanced Shahidi Conjecture has been proved for the symplectic and split odd special orthogonal groups in~\cite{HLL,HLLZ}.  
Assuming the \( p \)-adic Kazhdan-Lusztig hypothesis, a \( p \)-adic version of Theorem~\ref{theo:ESC} is established in~\cite{cunninghamWhittaker,cunninghamBalodis}.


\section{Micro Packets}
In this section we recall the definition of a micro–packet given in \cite[Definition~19.15]{ABV}.  
For this purpose, we first review the reformulation and enhancement of the Langlands correspondence presented in \cite[Theorem~10.4]{ABV}.  After the notion of micro–packet has been introduced, we conclude 
with a brief overview of the translation principle  
and its role in reducing  
the description of micro–packets with singular infinitesimal character to the case of regular infinitesimal character. 

To better present the theory developed in \cite{ABV}, we adopt their framework.  
That is, rather than fixing a single real form $G(\mathbb{R})$ and studying its representations, 
we consider all pure real forms appearing in the inner class of $G(\mathbb{R})$ and study their corresponding 
sets of irreducible representations simultaneously.
This will also simplify the exposition of the connections between representations and the geometric objects studied in \cite{ABV}.

Let \(G\) be a connected reductive complex algebraic group. 
We fix an inner class of real forms of \(G\) (\cite[Equation 2.4]{ABV}). This inner class is determined by a unique quasisplit 
real form \(\delta_q\) fixing a pinning 
\begin{equation}
  \label{pinning}
  \left(B,T,\{X_{\alpha}\}\right),
\end{equation}
where \(B\) is a Borel subgroup, \(T \subset B\) is a maximal torus, and
\(\{X_{\alpha}\}\) is a set of simple root vectors relative to the positive root system
\(R^{+}(G,T) = R(B,T)\) of \(R(G,T)\). We will also consider the Langlands dual group \( {}^{\vee}G \), whose root datum 
is dual to that of \( G \):
\[
{}^{\vee}G \supset {}^{\vee}B \supset {}^{\vee}T,\quad 
R^{+}({}^{\vee}G, {}^{\vee}T) = \{ \alpha^{\vee} :\, \alpha \in R^{+}(G,T) \},\
\]
Starting from the quasisplit real form $\delta_q$, we enlarge \(G\) to define the extended group
\begin{equation}\label{eq:extendedgroup}
  G^\Gamma = G \rtimes \langle \delta_q \rangle,
\end{equation}
and, following  \cite[Definition~2.13]{ABV}, we call an element $\delta \in G^\Gamma \,-\, G$ such that $\delta^2 \in Z(G)$ has finite order 
a \emph{strong real form} of $G^\Gamma$.  
Since our main result only pertains to the quasi-split real form $\delta_q$, 
in these notes we restrict our attention to the strong real forms 
$\delta \in G^\Gamma$ satisfying $\delta^2 = 1$.  
These are the \emph{pure strong real forms} \cite[\S2]{vogan_local_langlands}.
To each such $\delta$, we denote by $G(\mathbb{R},\delta)$  the group
of points in \(G\) fixed by \(\delta\). 


Putting together the partitions provided by the Langlands correspondence for the different sets  
\( \Pi(G(\mathbb{R}, \delta)) \), we recover a partition of the set \( \Pi(G/\mathbb{R}) \) of irreducible admissible representations  of $G$-conjugacy classes of pure strong real forms of \( G^{\Gamma} \)~\cite[Lemma~1.15]{ABV}. In this manner,  
we extend \eqref{eq:LLC} by attaching to each L-parameter $\varphi\in\Phi({}^{\vee}G^{\Gamma})$ a \emph{compound packet},
 which we continue to call an L-packet
\[
\Pi_{\varphi}(G/\mathbb{R}) := \bigsqcup_{\delta\in \Sigma} \Pi_\varphi(G(\R,\delta)),
\]
where $\Sigma$ is a set of representatives for the $G$-conjugacy classes 
of pure strong real forms of \(G^{\Gamma}\).
Let us move on to explore the geometric reformulation of this correspondence developed in \cite{ABV}.  
For this, we fix a semisimple element $\lambda \in {}^{\vee}\mathfrak{g}$. After
conjugating by $\ch G$ we may assume $\lambda\in {}^{\vee}\mathfrak{t}$.  Using the
canonical isomorphism ${}^{\vee}\mathfrak{t}\simeq{}^{\vee}\mathfrak{t}^*$ we identify $\lambda$ with an
element of ${}^{\vee}\mathfrak{t}^*$, and hence via the Harish-Chandra homomorphism, with
an infinitesimal character for $G$. This construction only depends on
the $\ch G$-orbit of $\lambda$. We therefore refer to a semisimple element
$\lambda\in\ch \g$, or to its $\ch G$-orbit   
\begin{align}\label{eq:infchar}
\mathcal{O} = {}^{\vee}G \cdot \lambda,
\end{align}
as an \emph{infinitesimal character} for $G$.

We denote by ${}^{\vee}G(\lambda)$ the centralizer of $e(\lambda)=\exp(2\pi i \lambda)$ in ${}^{\vee}G$, and by ${}^{\vee}L(\lambda)$ the centralizer of $\lambda$ in ${}^{\vee}G$.  
In addition, we let ${}^{\vee}\mathfrak{n}(\lambda)$ be the sum of the positive integral eigenspaces of $\operatorname{ad}(\lambda)$, and ${}^{\vee}N(\lambda)$ the connected unipotent subgroup with Lie algebra ${}^{\vee}\mathfrak{n}(\lambda)$.  
These subgroups combine to define the parabolic subgroup
\begin{align}\label{eq:parabolic}
{}^{\vee}P(\lambda) = {}^{\vee}L(\lambda) {}^{\vee}N(\lambda),
\end{align}
of ${}^{\vee}G(\lambda)$.  
Consider the set
\[
\mathcal{I}(\lambda) = \left\{\, y \in {}^{\vee}G^{\Gamma} \,-\, {}^{\vee}G \;:\; 
 y^{2} = e(\lambda) \,\right\}.
\]  
By \cite[Proposition~6.13]{ABV}, this set contains only finitely many ${}^{\vee}G(\lambda)$–orbits,  
which we denote by
\[
\mathcal{I}_{1}(\lambda),\, \dots\, , \mathcal{I}_{r}(\lambda).
\]
For each $1 \leq i \leq r$, choose a representative $y_{i} \in \mathcal{I}_{i}(\mathcal{O})$  
satisfying $y_{i}^{2} = e(\lambda)$.  
Conjugation by $y_{i}$ then defines an involution of ${}^{\vee}G(\lambda)$,  
whose fixed point set we denote by 
\begin{align}\label{eq:K}
{}^{\vee}K_{i}\, =\, {}^{\vee}G(\lambda)^{y_i}.
\end{align}
For each $1 \leq i \leq r$, define the variety
\begin{align}\label{eq:varietiesisomor}
X_{i}\left(\mathcal{O}, {}^{\vee}G^{\Gamma}\right) \;=\; {}^{\vee}G \times_{K_{i}} {}^{\vee}G(\lambda)/P(\lambda).
\end{align} 
We then combine these varieties to form
\begin{align}\label{eq:varietiesisomor2}
X\left(\mathcal{O}, {}^{\vee}G^{\Gamma}\right)\;=\; \bigsqcup_{i=1}^{r} {}^{\vee}G \times_{{}^{\vee}K_i} X_i\left(\mathcal{O}, {}^{\vee}G^{\Gamma}\right).
\end{align}
The ${}^{\vee}G$–orbits on 
$X_{i}\bigl(\mathcal{O}, {}^{\vee}G^{\Gamma}\bigr)$  
are in one-to-one correspondence with the $K_{i}$–orbits on the partial flag variety  
${}^{\vee}G(\lambda)/P(\lambda)$.  
Moreover, by \cite[Proposition~7.14]{ABV}, this correspondence preserves both the closure relations  
and the nature of the singularities of the orbit closures.
Finally, the union of the varieties $X\left(\mathcal{O}, {}^{\vee}G^{\Gamma}\right)$, as $\mathcal{O}$ runs over the ${}^{\vee}G$–orbits of semisimple elements in ${}^{\vee}\mathfrak{g}$,  
yields the variety $X\left({}^{\vee}G^{\Gamma}\right)$ mentioned in the introduction and used in \cite{ABV} to reformulate the Langlands correspondence.

In a little more detail, by \cite[Proposition~5.6]{ABV}, the set of 
L-homomorphisms $P\left({}^{\vee}G^{\Gamma}\right)$  
can be identified with the set 
of pairs $(y', \lambda')$, where $\lambda' \in {}^{\vee}\mathfrak{g}$ is semisimple, $(y')^2=e(\lambda')$, and
\[
[\lambda', \operatorname{Ad}(y')\lambda'] = 0.
\]
We denote by $\varphi(y', \lambda')$ the L–homomorphism corresponding to the pair $(y', \lambda')$. Consider the set
\[
P\left(\mathcal{O}, {}^{\vee}G^{\Gamma}\right) = \left\{\, \varphi(y', \lambda') \in 
P\left({}^{\vee}G^{\Gamma}\right) : \lambda' \in \mathcal{O} \,\right\}
\]
of L–homomorphisms with infinitesimal character $\mathcal{O} = {}^{\vee}G \cdot \lambda$~\eqref{eq:infchar}, and write
\(
\Phi\left(\mathcal{O}, {}^{\vee}G^{\Gamma}\right)
\)
for the set of ${}^{\vee}G$–conjugacy classes in $P\left(\mathcal{O}, {}^{\vee}G^{\Gamma}\right)$. 
Then, as a representative of any
$\varphi \in \Phi\left(\mathcal{O}, {}^{\vee}G^{\Gamma}\right)$, we may choose
the L–homomorphism $\varphi(y', \lambda)$, where $y' \in \mathcal{I}(\lambda)$. 
Since $y' \in \mathcal{I}(\lambda)$,
there exists $1 \leq i \leq r$ and $g\in {}^{\vee}G(\lambda)$ such that $y'=gy_ig^{-1} \in K_i$. Then, by
\cite[Propositions~6.16 and 6.17]{ABV}, the identification
\begin{align}\label{eq:KorbitToLmorphism}
\text{${}^{\vee}G$–conjugacy class of } \varphi(y', \lambda') \;\longmapsto\; {}^{\vee}K_i \cdot (g^{-1} \, {}^{\vee}P(\lambda))
\end{align}
induces a bijection
\begin{align}\label{eq:KorbitToLmorphismBij}
\Phi\left(\mathcal{O}, {}^{\vee}G^{\Gamma}\right) \quad \longleftrightarrow\quad \bigcup_{i=1}^r\left\{{}^{\vee}K_i\text{–orbits in } {}^{\vee}G(\lambda) / {}^{\vee}P(\lambda)\right\},
\end{align}
which, by Equality \eqref{eq:varietiesisomor2}, can be seen as a bijection between the
${}^{\vee}G$–orbits of $X\left(\mathcal{O}, {}^{\vee}G^{\Gamma}\right)$ and the set of L-parameters $\Phi\left(\mathcal{O}, {}^{\vee}G^{\Gamma}\right)$.

We now fix one of the fixed point sets ${}^{\vee}K_i$ (with $1 \leq i \leq r$), 
denote it simply by ${}^{\vee}K$, and restrict our attention to the set 
${}^{\vee}K \backslash X(\lambda)$ of ${}^{\vee}K$–orbits in the partial flag variety
$$X(\lambda)={}^{\vee}G(\lambda) / {}^{\vee}P(\lambda).$$
Through the bijection \eqref{eq:KorbitToLmorphism},  
the Local Langlands Correspondence assigns to each ${}^{\vee}K$–orbit  
$S \subset {}^{\vee}G(\lambda) / {}^{\vee}P(\lambda)$ an L–packet,  
\begin{align}\label{eq:LtoS}
S \longleftrightarrow\ \varphi_S\ \longleftrightarrow\
\Pi_{S}(G/\R):=\Pi_{\varphi_S}(G/\R).
\end{align}
This correspondence is refined in \cite{ABV} by enriching the ${}^{\vee}K$–orbit  
$S$ with a ${}^{\vee}K$–equivariant local system $\mathcal{V}$ of complex vector spaces on $S$.  
A complete geometric parameter is then a pair $\xi = (S, \mathcal{V})$,  
and the enhanced local Langlands correspondence, stated in \cite[Theorem~10.4]{ABV},  
associates to $\xi$ a unique (infinitesimal equivalence class of an) irreducible admissible representation $\pi(\xi)$ of a pure real form of $G$, with infinitesimal character $\lambda$.
In particular, we obtain a bijection
\begin{equation}\label{eq:RLLC}
 \pi(\xi)\ \longleftrightarrow\ \xi=(S,\mathcal{V}),
\end{equation}
between the set of irreducible local systems supported on $S$ and the extended L–packet  
$ \Pi_{S}(G/\R)$.
The set of complete geometric parameters will be denoted by
\begin{equation}\label{eq:complete-geoparam}
\Xi(X(\lambda),{}^{\vee}K).
\end{equation}
A striking feature of \eqref{eq:RLLC} is that the pair
$\xi = (S, \mathcal{V})$ determines an irreducible perverse sheaf
$P(\xi)$ on $X(\lambda)$. 
Let 
\(j: S \longrightarrow \overline{S}\)
denote the inclusion of $S$ into its closure, and
\(i: \overline{S} \longrightarrow X(\lambda)\)
the inclusion of the closure $\overline{S}$ into the flag variety $X(\lambda)$.  
If we regard the local system $\mathcal{V}$ as a constructible sheaf on $S$,  
the complex $\mathcal{V}[-\dim S]$, consisting of the single sheaf $\mathcal{V}$ placed in degree $-d$,  
defines a  ${}^{\vee}K$––equivariant perverse sheaf on $S$.  
Applying the intermediate extension functor $j_{!\ast}$, followed by the direct image functor $i_{\ast}$,  
we obtain an irreducible perverse sheaf on $X(\lambda)$,
\begin{equation}\label{eq:irredPerv}
P(\xi) = i_{\ast} j_{!\ast} \mathcal{V}[-\dim S],
\end{equation}
called the perverse extension of $\mathcal{V}$.  
By then applying the de Rham functor, we obtain an irreducible $\mathcal{D}_{X(\lambda)}$–module
\begin{equation}\label{eq:irredDM}
D(\xi) = DR^{-1}(P(\xi)).
\end{equation}
Everything is now in place to turn to the definition of micro-packets.  
The construction hinges on two invariants associated with an irreducible $\mathcal{D}_{X(\lambda)}$–module $D = D(\xi)$.

The first invariant is the \emph{characteristic variety} $\mathrm{Ch}(D)$ \cite[Section~19]{ABV}, \cite[Sections~2.1--2.2]{Hotta}.  
It is a closed algebraic subvariety of the cotangent bundle $T^{\ast}X(\lambda)$, 
In fact, it is contained in the conormal bundle $T^{\ast}_{{}^{\vee}K} X(\lambda)$ to the ${}^{\vee}K$–action on $X(\lambda)$ 
\cite[(19.1), Proposition~19.12(b)]{ABV}, and its  ${}^{\vee}K$–\textbf{components}—that is, the smallest ${}^{\vee}K$–invariant unions of irreducible components—of 
$T^{\ast}_{{}^{\vee}K} X(\lambda)$ are the closures $\overline{T_S^{\ast} X(\lambda)}$ of the conormal bundles to 
${}^{\vee}K$–orbits $S$ in $X$ \cite[Lemma~19.2(b)]{ABV}.  
According to \cite[Proposition~19.12(c)]{ABV}, the characteristic variety $\mathrm{Ch}(D)$ is a union of such ${}^{\vee}K$–components.  
To every irreducible component $C$ of $\mathrm{Ch}(D)$, there is an associated local module whose length $m_C(D)$ \cite[Definition~19.7]{ABV} depends only on the ${}^{\vee}K$–component $\overline{T_S^{\ast} X(\lambda)}$ containing $C$; in other words, all irreducible components within a given ${}^{\vee}K$-component have the same length. 
This leads to the definition of the \emph{characteristic cycle} of $D$ as a formal sum
\begin{equation}
\label{cc}
CC(D) = \sum_{C} m_C(D) \, C 
\;=\; \sum_{S \in {^\vee}K \backslash X(\lambda)} \chi_S^{\mathrm{mic}}(D)\ \overline{T^*_{S} X(\lambda)},
\end{equation}
where $\chi_S^{\mathrm{mic}}(D) = m_C(D)$ for any irreducible component $C$ contained in the 
${}^{\vee}K$–component corresponding to $\overline{T^*_{S} X}$.  
The non-negative integer $\chi_S^{\mathrm{mic}}(D)$ is called the \emph{microlocal multiplicity along $S$}.  

By \cite[Theorem~2.2.3]{Hotta}, taking characteristic cycles defines a $\mathbb{Z}$–linear map
\begin{equation}
\label{eq:CC-map}
CC: \mathscr{K}(X(\lambda), {^\vee}K) \longrightarrow \mathscr{L}(X(\lambda), {^\vee}K).
\end{equation}
The domain of the map is
the Grothendieck group $\mathscr{K}(X(\lambda),{^\vee}K)$ of the category of
${^\vee}K$-equivariant regular holonomic sheaves of
$\mathcal{D}$-modules on $X(\lambda)$.  The codomain of the map is
\begin{equation}
\label{LXK} 
\mathscr{L}(X(\lambda),{^\vee}K)=\left\{\sum_{S\in {^\vee}K\setminus
  X(\lambda)}m_S\ \overline{T^{\ast}_{S}X(\lambda)}:m_S\in\mathbb{Z}\right\}
\end{equation}
the free $\mathbb{Z}$–module generated by the closures of conormal bundles $\overline{T^*_{S} X}$, with $S \in {^\vee}K \backslash X(\lambda)$.  
Since, via the bijection \eqref{eq:irredDM}, we have an isomorphism between $\mathscr{K}(X, {^\vee}K)$ and the Grothendieck group of the category of
${^\vee}K$–equivariant perverse sheaves on $X$, we may identify these two Grothendieck groups, so that we can define
\begin{equation}\label{eq:D=P}
\chi_S^{\mathrm{mic}}(P(\xi)) =\chi_S^{\mathrm{mic}}(D(\xi)) 
\quad\text{ and }\quad
CC(P(\xi)) = CC(D(\xi)).
\end{equation}
We will be using these maps on both irreducible perverse sheaves and irreducible $D_{X(\lambda)}$-modules, depending on the context.

We are now in a position to define the micro-packet attached to a ${}^{\vee}K$–orbit $S \subset X(\lambda)$.  
The micro-packet of $S$ is defined as
\begin{equation}\label{mpacket}
\Pi^{\mathrm{mic}}_{S}(G/\R) \;=\; \{\, \pi(\xi')\in\Pi(G/\R)  : \chi_S^{\mathrm{mic}}(D(\xi')) \neq 0 \,\},
\end{equation}  
\cite[Definition~19.13]{ABV}. If the L–parameter $\varphi$ corresponds under \eqref{eq:LtoS} to the ${}^{\vee}K$–orbit $S_\varphi$,  
we denote the corresponding micro-packet by
\[
\Pi^{\mathrm{mic}}_\varphi(G/\R) := \Pi^{\mathrm{mic}}_{S_\varphi}(G/\R).
\]
This is a set of irreducible representations of pure strong involutions of $G$.
We are primarily interested in the packet for the quasisplit strong
involution. For this, given any pure strong involution $\delta \in G^{\Gamma}$, we define
\begin{equation}\label{mpacketdelta}
\Pi^{\mathrm{mic}}_{S}(G(\R,\delta)) \;=\; 
\bigl\{\, \pi(\xi') \in \Pi(G(\R,\delta)) : \chi_S^{\mathrm{mic}}(D(\xi')) \neq 0 \,\bigr\}.
\end{equation}
The representations in the packet $\Pi_S^{\mathrm{mic}}(G/\R)$ all have infinitesimal character 
$\lambda$. The proof of the Enhanced Shahidi Conjecture 
is carried out by reducing to the case where $\lambda$ is regular, i.e.
\begin{equation}\label{eq:regular}
{^\vee}\alpha(\lambda) \in \mathbb{Z} \ \Longrightarrow \
{^\vee}\alpha(\lambda) > 0,  \quad {^\vee}\alpha \in R\left({^\vee}B, {^\vee}T\right).
\end{equation}
The key tool that allows this reduction is the \emph{translation principle}.  
We conclude this section with a brief review of its main properties.  
In particular, we will focus on the interaction between the translation principle  
and the microlocal multiplicities, as this will allow us to reduce the description of  
micro-packets with singular infinitesimal character to the regular case.

The translation principle in representation theory was first introduced by Jantzen \cite{Jantzen}  
and independently by Vogan and Zuckerman \cite[Chapter 7]{greenbook}.  
A geometric version of this principle was later developed in \cite[Chapter~8]{ABV}.  
We will assume that the reader is already familiar with the representation-theoretic formulation,  
and concentrate here on its geometric incarnation.

The translation principle begins with the existence of a regular
element $\lambda' \in  {}^{\vee}\mathfrak{t}$, with 
\({^\vee}\mathrm{G}\)-orbit \(\mathcal{O}' \subset
  {^\vee}\mathfrak{g}\),  and a
  \emph{translation datum} \(\mathcal{T}\) from \(\mathcal{O}\) to
  \(\mathcal{O}'\) (\cite[Definition 8.6, Lemma 8.7]{ABV}). Two
  requirements of a translation datum are that
\begin{equation}\label{eq:lambdaprime}
\lambda - \lambda' \in X_*( {^\vee}T),
\end{equation}
and that for every root $\alpha \in
R({}^{\vee}\mathrm{G}, {}^{\vee}T)$, we have
\begin{align}\label{eq:translationdatum1}
\alpha(\lambda) \in \mathbb{N} \,-\, \{0\} 
\quad \Longrightarrow \quad 
\alpha(\lambda') \in \mathbb{N} \,-\, \{0\}.
\end{align}
Define \({}^{\vee}P(\lambda)\) and \({}^{\vee}P(\lambda')\) as in \eqref{eq:parabolic}. Then due to Equations \eqref{eq:lambdaprime} and \eqref{eq:translationdatum1}, we have the containment
\[
{}^{\vee}P(\lambda') \subset {}^{\vee}P(\lambda).
\]
Recall from the begining of this section that ${^\vee}\mathrm{G}(\lambda)$ is the centralizer of $e(\lambda)$ in ${}^{\vee}\mathrm{G}$, and that ${^\vee}K$ is the centralizer of an element $y \in {^\vee}\mathrm{G}(\lambda')$ with $y^2 = e(\lambda)$. It follows from Equations \eqref{eq:lambdaprime} and \eqref{eq:translationdatum1} that $e(\lambda) = e(\lambda')$, which in turn implies that the corresponding centralizers coincide, that is
\[
{^\vee}{G}(\lambda') = {^\vee}{G}(\lambda) \quad \text{and} \quad
{^\vee}K(\lambda') ={^\vee}{G}(\lambda')^y= {^\vee}K(\lambda).
\]
Despite these equalities, we will persist in  writing 
${^\vee}\mathrm{G}(\lambda)$ and
${^\vee}\mathrm{G}(\lambda')$, and likewise  
${^\vee}K(\lambda)$ and ${^\vee}K(\lambda')$, so that the reader can
easily distinguish when the theory pertains to the singular or to the regular
setting.  

We continue to denote by \(X(\lambda)\) the generalized flag variety \({}^{\vee}{G}(\lambda) / {}^{\vee}P(\lambda)\) of \({}^{\vee}{G}(\lambda)\), and by \(X(\lambda')\) the flag variety \({}^{\vee}G(\lambda') / {}^{\vee}P(\lambda')\). Then, by \cite[Proposition 8.8]{ABV}, the surjection
\begin{align}\label{eq:ftee}
f_\mathcal{T}: X(\lambda') \longrightarrow X(\lambda)
\end{align}
defines smooth and proper morphism, with connected fibres of fixed dimension, which we denote by \(d\).
 According to \cite[Proposition 7.15]{ABV}, the morphism
$f_\mathcal{T}$
induces an inclusion
\begin{equation}\label{eq:fstar}
f^{*}_{\mathcal{T}}: \Xi(X(\lambda), {}^{\vee}K(\lambda)) \hookrightarrow \Xi(X(\lambda'), {}^{\vee}K(\lambda))
\end{equation}
of complete geometric parameters. We refer the reader to \cite[page 95]{ABV} for a detailed description of the image
\[
f_\mathcal{T}^{\ast}(\xi) = (f_\mathcal{T}^{\ast}S, f_\mathcal{T}^{\ast}\tau)
\]
of a complete geometric parameter \(\xi = (S, \tau) \in \Xi(X(\lambda), {}^{\vee}K(\lambda))\). We only mention here that, since \({}^{\vee}K(\lambda')\) has only finitely many orbits on \(X(\lambda')\), and \(f^{\ast}_\mathcal{T}\) has connected fibres, there is a unique open orbit of \({}^{\vee}K(\lambda')\) in \(f_{\mathcal{T}}^{-1}(S)\). The orbit \(f_\mathcal{T}^{\ast}S\) is defined to be this unique open orbit.

That being said, we can now introduce the geometric version of the translation functor in representation theory.  
As defined in \cite[Proposition~8.8(b)]{ABV}, this functor maps the category 
 $\mathcal{P}\left(
X(\lambda),{}^{\vee}K(\lambda)\right)$
of \({}^{\vee}K(\lambda)\)–equivariant perverse sheaves on \(X(\lambda)\) to the category 
 $\mathcal{P}\left(
X(\lambda'),{}^{\vee}K(\lambda')\right)$
of \({}^{\vee}K(\lambda')\)–equivariant perverse sheaves on \(X(\lambda')\).  
It is given by the inverse image along the projection in \eqref{eq:ftee}, shifted by the relative dimension \(d\)
\begin{align}\label{eq:geometricTF}
f^{*}_{\mathcal{T}}[d]: \mathcal{P}\left(
X(\lambda),{}^{\vee}K(\lambda)\right) \longrightarrow \mathcal{P}\left(
X(\lambda'),{}^{\vee}K(\lambda')\right).
\end{align}
Following \cite[Proposition 7.15(b)]{ABV}, this is a fully faithful exact functor, sending irreducible perverse sheaves to irreducible perverse sheaves. Furthermore, it satisfies
\[
f^{\ast}_{\mathcal{T}}[d]\left(P(\xi)\right) = P(f_\mathcal{T}^{\ast}(\xi)).
\]
Another important property of the translation functor is that it preserves the microlocal multiplicities along the \({}^{\vee}K(\lambda)\)-orbits of \(X(\lambda)\). 
Indeed, \cite[Proposition 20.1(e)]{ABV} shows that for all $\xi \in \Xi(X(\lambda), {}^{\vee}K(\lambda))$ and for all ${}^{\vee}K(\lambda)$–orbits $S$ in $X(\lambda)$, we have
\begin{align}
\begin{aligned}\label{eq:translationofchiV1}
\chi_{S}^{\mathrm{mic}}(P(\xi)) = \chi_{f^{\ast}_{\mathcal{T}}S}^{\mathrm{mic}}(f^{\ast}_{\mathcal{T}}P(\xi)) = \chi_{f^{\ast}_{\mathcal{T}}S}^{\mathrm{mic}}(P(f^{\ast}_{\mathcal{T}}\xi)).
\end{aligned}
\end{align}
As an immediate consequence, the characteristic cycle of \(P(\xi)\) can be expressed as
\begin{align}\label{eq:translationofchiV2}
CC(P(\xi)) = \sum_{S} \chi_{f^{\ast}_{\mathcal{T}}S}^{\mathrm{mic}}(P(f^{\ast}_{\mathcal{T}}\xi)) \cdot \overline{T^{\ast}_{S}X(\lambda)}.
\end{align}
In addition, by the definition of a micro-packet \eqref{mpacket}, we have
\begin{eqnarray}\label{eq:singularmicpacketrelation}
\pi(\xi)\in\Pi_{S}^{\mathrm{mic}}
\quad \Longleftrightarrow\quad 
\pi(f^{*}_{\mathcal{T}}\xi)\in\Pi_{f^{*}_{\mathcal{T}}S}^{\mathrm{mic}}.
\end{eqnarray}
To complete our description, we need to relate the geometric translation functor \eqref{eq:geometricTF} with its representation-theoretic counterpart.  
The (Jantzen–Zuckerman) translation functor in representation theory \cite[(17.8j)]{AvLTV} is an exact functor
\begin{equation}
  \label{transfunct}
  \Psi_{\mathcal{T}}:\,  \Pi\bigl(\mathcal{O}', \mathrm{G}/\mathbb{R} \bigr)
  \longrightarrow \Pi\bigl(\mathcal{O}, \mathrm{G}/\mathbb{R}\bigr),
\end{equation}
from the category of representations of pure real forms of \(\mathrm{G}\) with infinitesimal character \(\mathcal{O}'\), to the category of representations of pure real forms of \(\mathrm{G}\) with infinitesimal character \(\mathcal{O}\).
According to \cite[Corollary 17.9.8]{AvLTV}, the functor \(\Psi_{\mathcal{T}}\) is surjective. Moreover, by \cite[Proposition 16.4(b)]{ABV}, for any complete geometric parameter \(\xi \in \Xi(X(\lambda'), {}^{\vee}K(\lambda'))\), the image \(\Psi_{\mathcal{T}}(\pi(\xi))\) is either irreducible or zero—the former occurring if and only if \(\xi = f^{\ast}_\mathcal{T}(\gamma)\) for some \(\gamma \in \Xi(X(\lambda), {}^{\vee}K(\lambda'))\).

By \cite[Proposition 16.6]{ABV}, we have the identity
\begin{align}\label{eq:SingularPimic}
\Psi_{\mathcal{T}}(\pi(f^{\ast}_\mathcal{T}(\xi))) = \pi(\xi).
\end{align}
This equality, together with \eqref{eq:singularmicpacketrelation}, imply that 
 the image of the micro-packet of
\(f^{\ast}_\mathcal{T}S\)  under the translation functor is equal to
the micro-packet of \(S\)  
\begin{align}\label{eq:Singularetamic}
\Pi_{S}^{\mathrm{mic}}(G/\mathbb{R})\ =\ \left\{\Psi_\mathcal{T}(\pi):\ \pi\in 
\Pi_{f^{\ast}_\mathcal{T}S}^{\mathrm{mic}}(G/\R),\ \Psi_\mathcal{T}(\pi)\neq 0 \right\}.
\end{align}

\section{Weyl group actions and the  characteristic cycle map}

Let $\lambda$ be a semisimple element in ${}^{\vee}\mathfrak{g}$.  
In this section, we assume that $\lambda$ satisfies \eqref{eq:regular}.  
Let ${}^{\vee}G(\lambda)$ and ${}^{\vee}P(\lambda)$ be as in the paragraph preceding \eqref{eq:parabolic},
and consider the flag variety 
\(
X(\lambda) = {}^{\vee}G(\lambda)/{}^{\vee}P(\lambda).
\)  
Let ${}^{\vee}K$ be one of the fixed point sets introduced in \eqref{eq:K}.  
Our focus will be on the set ${}^{\vee}K \backslash X(\lambda)$ of ${}^{\vee}K$–orbits in $X(\lambda)$.

To prove that the micro-packet associated with the open orbit in 
$X(\lambda)$ is the only one containing generic representations, 
we will compute the characteristic cycles of the $D_{X(\lambda)}$-modules
corresponding, under \eqref{eq:RLLC} and \eqref{eq:irredDM}, to this class of representations.  Although  computing characteristic cycles is generally a highly nontrivial task, 
in this specific case the situation simplifies: it suffices to analyze the behavior of their characteristic cycles 
under the action of the Weyl group
$$
{}^{\vee}W = W({^\vee}G, {^\vee}T).
$$
Let us briefly review how this action is realized. In the representation-theoretic setting, there is a natural ${}^{\vee}W$--action on $\mathscr{K}(X(\lambda), {^\vee}K)$, commonly referred to as the \emph{coherent continuation representation} \cite[Definition~7.2.28]{greenbook}.  
This ${}^{\vee}W$-action can be explicitly computed using the \textsf{Atlas of Lie Groups and Representations} software \cite{atlas}.  

On the geometric side, there is a ${}^{\vee}W$-action on $\mathscr{L}(X(\lambda), {^\vee}K)$, as 
described in \cite{Hotta85} and \cite{Rossmann95}. By the work of Tanisaki \cite{Tanisaki}, this action is compatible with the coherent continuation representation on $\mathscr{K}(X, {^\vee}K)$.  
More precisely, by \cite[Theorem~1]{Tanisaki} 
 the characteristic cycle map 
 is ${}^{\vee}W$-equivariant, \emph{i.e.},
\begin{equation}\label{eq:W-equivariant}
CC(w \cdot D) = w \cdot CC(D), \quad D \in \mathscr{K}(X, {^\vee}K), \ w \in {}^{\vee}W.
\end{equation}
In the next section, we will exploit this ${}^{\vee}W$-equivariance  
to compute the characteristic cycle of the irreducible $D_{X(\lambda)}$-modules  
corresponding to irreducible generic representations. For the time being, let us say a few words on how ${}^{\vee}W$ acts on the domain and codomain of $CC$.

We start by considering its action on $\mathscr{K}(X(\lambda), {^\vee}K)$. To this end, we look at the Grothendieck group 
$\mathscr{K}({}^{\vee}\mathfrak{g},{^\vee}K)$
of the category of $({^\vee}\mathfrak{g}, {^\vee}K)$-modules with infinitesimal character equal to the half-sum of the roots in $R({^\vee}B,{^\vee}T)$. By the Beilinson–Bernstein correspondence (\cite[Theorem 8.3]{ABV}, \cite[Sections 11.5–11.6]{Hotta}), this Grothendieck group can be identified with $\mathscr{K}(X(\lambda, {^\vee}K)$.

The coherent continuation representation, first introduced by Vogan-Zuckerman, is originally defined as an action of ${}^{\vee}W$ on $\mathscr{K}({}^{\vee}\mathfrak{g},{^\vee}K)$.  
Via the 
Beilinson-Bernstein correspondence this action transfers to an action of ${}^{\vee}W$ on $\mathscr{K}(X(\lambda), {^\vee}K)$.  
For a precise definition and a detailed description of its properties, we refer the reader to \cite[Definition~7.2.28 and Lemma~7.2.29]{greenbook}.  
For a description of the action of simple reflections on the basis given by irreducible representations, see \cite[Lemma~14.7]{ICIV}.  
From these references, the notion we need is the \emph{$\tau$-invariant} (following \cite[Definition~7.3.8]{greenbook}).  
For an irreducible $D_{X(\lambda)}$-module $D \in \mathscr{K}(X(\lambda), {^\vee}K)$, we say that a simple root ${^\vee}\alpha\in R^+({}^{\vee}G,{}^{\vee}T)$ is in the (Borho-Jantzen-Duflo) $\tau$-invariant $\tau(D)$ of $D$ if
\begin{equation}\label{eq:verticalP}
s_{{}^{\vee}\alpha}^{} \cdot D = -D,
\end{equation}
where $s_{{}^{\vee}\alpha}^{} \in {}^{\vee}W$ is the simple reflection corresponding to ${}^{\vee}\alpha$.  
The construction defining the ${}^{\vee}W$-action on $\mathscr{K}({}^{\vee}\mathfrak{g}, {^\vee}K)$ also induces an action of the Weyl group
$W = W(G,T)$, on the Grothendieck group of the category of representations of pure real forms of \(G\) with infinitesimal character $\lambda$.  
In particular, this provides a definition of the $\tau$-invariant for irreducible representations of pure real forms of \(G\).  
Let $\xi\in\Xi(X(\lambda), {}^{\vee}K)$ be a complete geometric parameter. 
What matters for us concerning the $\tau$-invariant 
 is the relationship between 
$\tau(\pi(\xi))$ and 
$\tau(D(\xi))$. 
It follows from \cite[Corollary~14.9(b)]{ICIV} that
$$
s_\alpha \cdot \pi(\xi)\neq-\pi(\xi)
\quad\Longleftrightarrow\quad s_{{}^{\vee}\alpha} \cdot D(\xi)=-D(\xi).
$$ 
Consequently, if we write $\Delta\left({}^{\vee}B,{}^{\vee}T\right)$ for the set of simple roots in 
$R^+({}^{\vee}G,{}^{\vee}T)$, we obtain
\begin{equation}\label{eq:tau-comparison}
\tau(D(\xi)) \;=\; \Delta\left({}^{\vee}B,{}^{\vee}T\right) \;-\; \left\{{}^{\vee}\alpha : \alpha \in \tau(\pi(\xi))\right\}.
\end{equation}
We now turn to a brief review of the ${}^{\vee}W$-action on the codomain of $CC$.
There is a natural isomorphism between $\mathscr{L}(X(\lambda),{^\vee}K)$ and the top Borel–Moore homology
\[
H_{\mathrm{top}}^{\infty}(T_{{^\vee}K}^{\ast}X(\lambda), \mathbb{Z}).
\]  
Since the ${^\vee}K$-components of $T_{{^\vee}K}^{\ast}X(\lambda)$ are the closures of the conormal bundles of ${^\vee}K$-orbits $S$ in $X(\lambda)$, it follows from \cite[Lemma~19.1.1]{fulton} that the set
\(
\left\{[\overline{T_S^{\ast}X(\lambda)}] : S \in {^\vee}K \backslash X(\lambda) \right\}
\)
of fundamental classes of conormal bundle closures forms a basis for the top Borel-Moore homology
\[
H_{\mathrm{top}}^{\infty}(T_{{^\vee}K}^{\ast}X(\lambda), \mathbb{Z}) = \bigoplus_{S \in {^\vee}K \backslash X} \mathbb{Z} \cdot [\overline{T_S^{\ast}X(\lambda)}].
\]  
We identify each $\overline{T_S^{\ast}X}$ with its fundamental class $[\overline{T_S^{\ast}X}]$ and write
\[
\mathscr{L}(X(\lambda),{^\vee}K) = H_{\mathrm{top}}^{\infty}(T_{{^\vee}K}^{\ast}X(\lambda), \mathbb{Z}).
\]
In \cite{Hotta85}, \cite{HottaR} and \cite{Rossmann95}, the
topological construction of the Springer representation 
is expanded in order to define an action of the Weyl group
on the Borel-Moore homology space.  To make this action more concrete, we present in some detail how it is realized in the case of simple reflections.  
Let ${}^{\vee}\alpha\in \Delta\left({}^{\vee}B,{}^{\vee}T\right)$. Consider the natural projection
\begin{equation}
\label{pimap}
\pi_{{}^{\vee}\alpha}^{}: {}^{\vee}G(\lambda) / {}^{\vee}P(\lambda) \ \longrightarrow\  {}^{\vee}G(\lambda) / \left( {}^{\vee}P(\lambda) \sqcup  {}^{\vee}P(\lambda) s_{^{\vee}\alpha}^{}  {}^{\vee}P(\lambda)\right).
\end{equation}
Following \cite[Section~2]{Hotta85} and \cite[{\S3, Lemma~2}]{Tanisaki}, we say that a ${^\vee}K$-orbit $S$ is \emph{$s_{^{\vee}\alpha}^{}$-vertical} if, for every $x \in S$, the intersection
\begin{align}\label{eq:vertical0}
\pi_{{}^{\vee}\alpha}^{-1}(\pi_{{}^{\vee}\alpha}^{}(x)) \cap S
\end{align}
is open and dense in $\pi_{{}^{\vee}\alpha}^{-1}(\pi_{{}^{\vee}\alpha}^{}(x))$.  
If this condition is not satisfied, we say that $S$ is \emph{$s_{{}^{\vee}\alpha}^{}$-horizontal}.  
The following theorem provides a partial description of the ${}^{\vee}W$-action in this setting.
\begin{theo}[{\cite[Sections 3 and 4]{Hotta85}}, {\cite[Proposition
        8]{Tanisaki}}] 
\label{theo:Waction} Let $S$ be a ${}^{\vee}K$-orbit in $X(\lambda)$. Fix a simple root
${}^{\vee}\alpha^{} \in   \Delta\left({}^{\vee}B,{}^{\vee}T\right)$. 
\begin{enumerate}
\item 
If $S$ is $s_{{}^{\vee}\alpha}$-{\it{vertical}}, then
\begin{align}\label{eq:vertical}
s_{{}^{\vee}\alpha}\cdot \overline{T_S^{\ast}X(\lambda)}\, =\, -{\overline{T_S^{\ast}X(\lambda)}}. 
\end{align}
\item\label{theo:Waction2}  If $S$ is $s_{{}^{\vee}\alpha}$-{\it{horizontal}}, then
for any $s_{{}^{\vee}\alpha}$-vertical orbit
$S'\subset\overline{\pi_{{}^{\vee}\alpha}^{-1}(\pi_{{}^{\vee}\alpha}^{}(S))}$,  
there is a non-negative integer $n_{S, S'}$, such that
\begin{align}\label{eq:horizontal}
s_{{}^{\vee}\alpha} \cdot \overline{T^{*}_{S}(\lambda)}\, =\, \overline{T^{*}_{S}X(\lambda)}\, +\, 
\sum_{S'}n_{S,S'} \, \overline{T^{*}_{S'}X(\lambda)}.
\end{align}
\end{enumerate}
\end{theo}
We conclude this section with two results that will play a key role in the proof of Theorem~\ref{theo:ccforlarge} below.
The first is a direct consequence of the definition of $s_{{}^{\vee}\alpha}^{}$-horizontal and 
$s_{{}^{\vee}\alpha}$-vertical orbits.
The second, which exploits the ${}^{\vee}W$-equivariance of the characteristic cycle map, imposes restrictions on the conormal bundles that can appear in a characteristic cycle.
\begin{prop}\label{prop:horizontal}
Suppose that $S$ is a ${}^{\vee}K$-orbit in $X(\lambda)$ other than the open orbit. 
Then there exist a simple root ${}^{\vee}\alpha\in \Delta\left({}^{\vee}B,{}^{\vee}T\right)$ such that 
$S$ is $s_{{}^{\vee}\alpha}^{}$-horizontal.
\end{prop}
\begin{proof}
By \eqref{eq:vertical0}, for every simple root $\alpha \in \Delta\left({}^{\vee}B,{}^{\vee}T\right)$, we have 
\[
s_{{}^{\vee}\alpha}^{} \cdot \overline{T_S^{\ast}X} \;=\; -\overline{T_S^{\ast}X},
\]
if and only if $S$ is open in $\pi_{{}^{\vee}\alpha}^{-1}(\pi_{{}^{\vee}\alpha}^{}(S))$. 
Since the open orbit is the only orbit that satisfies this condition for all simple roots ${}^{\vee}\alpha$, the proposition follows.
\end{proof}
\begin{prop}{(\cite[Proposition 4.5]{ABM})}\label{prop:CC2}
Let $D \in \mathscr{K}(X(\lambda),{}^{\vee}K)$ be irreducible.
Suppose ${}^{\vee}\alpha \in \tau (D)$ (see \eqref{eq:verticalP}). If 
$\chi_{S}^{\mathrm{mic}} (D) \neq 0$ then $S$ is $s_{{}^{\vee}\alpha}^{}$-vertical. 
That is, all conormal bundles appearing in ${CC}(D)$ have
$s_{{}^{\vee}\alpha}^{}$-vertical orbits. 
\end{prop}

\section{The Enhanced Shahidi Conjecture}
Everything is now in place to verify the Enhanced Shahidi Conjecture.  
We start by computing the characteristic cycle of the $D$-module corresponding, under \eqref{eq:RLLC} and \eqref{eq:irredDM}, to an irreducible generic representation.  
The Enhanced Shahidi Conjecture then follows as a direct corollary from it.

We return to the setting of Section~2, removing the regularity assumption on the semisimple element $\lambda$ from the previous section.  
Recall that ${}^{\vee}G(\lambda)$ and ${}^{\vee}P(\lambda)$ were defined just before \eqref{eq:parabolic}, that ${}^{\vee}K$ is one of the fixed point sets introduced in \eqref{eq:K}, and that we denote by $X(\lambda)$ the flag variety ${}^{\vee}G(\lambda)/{}^{\vee}P(\lambda)$. We also write by $S_{\mathrm{open}}$ the open ${}^{\vee}K$-orbit in $X(\lambda)$.

Finally, we recall that a representation \(\pi\) is called \emph{generic} if it admits a Whittaker model. 
It follows from \cite{Kostant78} (see the discussion at the end of \cite{Vogan78}) 
that the irreducible generic representations are precisely those described in \cite[Theorem 6.2]{Vogan78}. 
This fact will play an important role in the proof of the next theorem.
\begin{theo}\label{theo:ccforlarge}
Let $\xi \in \Xi(X(\lambda),{}^{\vee}K)$ be a complete geometric parameter such that $\pi(\xi)$ is a generic representation. Then
\[
CC(D(\xi))\;=\; \overline{T_{S_\xi}^{\ast}X(\lambda)}  \;=\; \overline{T_{S_\mathrm{open}}^{\ast}X(\lambda)}.
\]
\end{theo}
\begin{proof}
Let $\xi \in \Xi(X(\lambda),{}^{\vee}K)$ be a complete geometric parameter such that $\pi(\xi)$ is generic.
We begin by treating the case where \(\lambda\) is regular.

By \cite[Theorem 6.2]{Vogan78}, the representation $\pi(\xi)$ is \emph{large}, 
that is, its $\tau$–invariant is empty:
\(
\tau(\pi(\xi)) = \emptyset.
\)
It then follows from \eqref{eq:tau-comparison}, that
\[
\tau(D(\xi)) = \Delta\left({}^{\vee}B,{}^{\vee}T\right).
\]
Therefore, by Proposition~\ref{prop:CC2}, the conormal bundle 
$\overline{T_S^{\ast}X(\lambda)}$ to the ${}^{\vee}K$–orbit $S\subset X(\lambda)$ appears in $CC(D(\pi))$  
if and only if $S$ is $\alpha$-vertical for all simple roots $\alpha$.  
Since, by Proposition~\ref{prop:horizontal}, this condition is satisfied only by  
the open orbit $S_{\mathrm{open}}$, we conclude that  
\begin{align}\label{eq:ccofgeneric}
CC(D(\xi))\;=\; \overline{T_{S_{\mathrm{open}}}^{\ast}X(\lambda)}.
\end{align}
To complete the proof, we remove the regularity assumption on the infinitesimal character, allowing the semisimple element 
\(\lambda \in {^\vee}\mathfrak{g}\) to correspond to a singular infinitesimal character. 
Fix \(\lambda'\) verifying equations \eqref{eq:lambdaprime} and
\eqref{eq:translationdatum1}, and consider the translation datum 
\(\mathcal{T}\) from the \({^\vee}G\)–orbit of \(\lambda\) to the \({^\vee}G\)–orbit of \(\lambda'\).

For any irreducible representation $\pi$ 
 we denote by  
$\dim \pi$ its Gelfand–Kirillov dimension. For the definition of $\dim \pi$  
and a detailed account of its properties, we refer the reader to \cite{Vogan78}.  
What is of interest to us here is that,  as explained in 
\cite[p.~84, after the proof of Corollary~4.7]{Vogan78}, it follows from 
\cite[Lemmas~2.3 and~2.4]{Vogan78} that, for any \(\xi \in \Xi(X(\lambda),{}^{\vee}K)\), one has
\begin{equation}\label{eq:GKdimEQUALITY}
\dim \pi(\xi) \, =\,  \dim \pi(f^{*}_{\mathcal{T}}\xi).
\end{equation}
In particular, 
$\dim \pi(\xi) \, =\,  |R^+(G,T)|$  
if and only if $\dim \pi(f^{*}_{\mathcal{T}}\xi)\, =\,  |R^+(G,T)|$. 
Therefore, by \cite[Theorem 6.2]{Vogan78}, the irreducible representation 
$\pi(\xi)$ is large, and hence generic, if and only if $\pi(f^{*}_{\mathcal{T}}\xi)$ is large, and hence generic.
Now, since $f^{*}_{\mathcal{T}}S_\mathrm{open}$ is open in $X(\lambda')$, it follows from \eqref{eq:ccofgeneric} that, if $\pi(f^{*}_{\mathcal{T}}\xi)$ is generic, we obtain  
\[
CC(D(f^{*}_{\mathcal{T}}\xi)) \;=\; \overline{T_{f^{*}_{\mathcal{T}}S_\mathrm{open}}^{\ast}X(\lambda')},
\]
and by Equation~\eqref{eq:translationofchiV2}, we deduce
\begin{align*}
CC(D(\xi)) 
&= \chi_{f^{\ast}_{\mathcal{T}}S_{\mathrm{open}}}^{\mathrm{mic}}(D(f^{\ast}_{\mathcal{T}}\xi)) 
   \cdot \overline{T^{\ast}_{S_{\mathrm{open}}}X(\lambda)} \\
&= \overline{T^{\ast}_{S_{\mathrm{open}}}X(\lambda)}.
\end{align*}
\end{proof}
As a corollary of Theorem~\ref{theo:ccforlarge},  
we obtain the Enhanced Shahidi Conjecture.

\begin{theo}\label{theo:ESC}
Let $S$ be a ${}^{\vee}K$–orbit in $X(\lambda)$. Then the micro-packet $\Pi_S^{\mathrm{mic}}(G/\R)$ has a generic member if and only if
\[
S = S_{\mathrm{open}}.
\]
In particular, for any A-parameter $\psi$, the A-packet $\Pi_\psi(G(\R,\delta_q))$ 
has a generic member if and only if $\psi|_{\mathrm{SL}_2}$ is trivial. 
\end{theo}

\begin{proof}
By definition, $\pi \in \Pi_S^{\mathrm{mic}}(G/\R)$ if and only if $\chi_{S}^{\mathrm{mic}}(\pi)\neq 0$.  
If, in addition, $\pi$ is generic, then by Theorem~\ref{theo:ccforlarge} we have  
$\chi_{S}^{\mathrm{mic}}(\pi)\neq 0$ if and only if $S = S_{\mathrm{open}}$.

Conversely, suppose $S$ is open. According to \cite[Lemma~19.14]{ABV},  
if $\pi \in \Pi_S^{\mathrm{mic}}(G/\R)$, then the ${}^{\vee}K$-orbit $S_\pi$ corresponding to L-parameter of $\pi$ satisfies $S \subset \overline{S_\pi}$.  
Therefore $S = S_\pi$ for all $\pi \in \Pi_S^{\mathrm{mic}}(G/\R)$,  
and the micro–packet $\Pi_S^{\mathrm{mic}}(G/\R)$ coincides with the L-packet $\Pi_S(G/\R)$. Moreover, as explained at the end of \cite[p.~19]{ABV},  
it follows from \cite[Proposition~1.11]{ABV}, \cite[Theorem~6.2]{Vogan78},  
and \cite{Kostant78} that the orbit $S$ is open if and only if  
the corresponding L-packet has a generic member.

Let us provide an alternative proof of the converse implication, 
which does not rely on \cite[Proposition 1.11]{ABV}. 
Suppose again that $S$ is open in $X(\lambda)$, and consider the complete geometric parameter 
\[
\xi = (S, \underline{\mathbb{C}}_{S}),
\] 
where $\underline{\mathbb{C}}_{S}$ denotes the constant sheaf on $S$. 
Since the closure of $S$ is smooth, it follows from the Decomposition Theorem for semi-small maps \cite[Theorem 3.4.1]{deCataldoMigliorini}, applied to the identity map, that
\[
P(\xi) = \underline{\mathbb{C}}_{X(\lambda)},
\]
Consequently, by \cite[Proposition~5.3.2]{Kashiwara-Schapira}, we have
\[
CC(P(\xi)) = \overline{T_S^{\ast} X(\lambda)}.
\]
Suppose for the moment that \(\lambda\) is regular. 
Since \(S\) is open,
it is \(s_{{}^{\vee}\alpha}^{}\)-vertical for all simple roots \({{}^{\vee}\alpha} \).
Hence, from 
\eqref{eq:W-equivariant} and \eqref{eq:vertical} , 
for all ${{}^{\vee}\alpha} \in \Delta\left({}^{\vee}B,{}^{\vee}T\right)$, we have
\begin{align*}
CC(s_{{}^{\vee}\alpha}^{}\ \cdot 	P(\xi)) 
  \, =\,  s_{{}^{\vee}\alpha}^{} \cdot CC(P(\xi)) 
 \,  =\,  -\,\overline{T_S^{\ast} X(\lambda)}.
\end{align*}
According to \cite[Lemma 14.7]{ICIV}, if \(s_{{}^{\vee}\alpha}^{} \cdot P(\xi) \neq -P(\xi)\),
then \(s_{{}^{\vee}\alpha}^{} \cdot P(\xi)\) is a linear combination with nonnegative integer coefficients of irreducible perverse sheaves. 
As a consequence, all conormal bundles must appear 
in \(CC(s_{{}^{\vee}\alpha}^{} \cdot P(\xi))\) with nonnegative coefficients. 
Since this is not the case, we deduce that
\(
s_{{}^{\vee}\alpha}^{} \cdot P(\xi) = -P(\xi)\), for all \({{}^{\vee}\alpha} \in \Delta\left({}^{\vee}B,{}^{\vee}T\right)\).
Therefore, \(\tau(P(\xi)) = \Delta\left({}^{\vee}B,{}^{\vee}T\right)\), and from \eqref{eq:tau-comparison} we obtain that 
\(\tau(\pi(\xi)) = \emptyset\). By \cite[Theorem~6.2]{Vogan78}, this implies that 
\(\pi(\xi)\) is generic.

We now proceed to remove the regularity hypothesis on $\lambda$. 
Fix a translation datum $\mathcal{T}$ from the \({^\vee}G\)-orbit of \(\lambda\) to the \({^\vee}G\)-orbit of a regular element $\lambda'$
satisfying
\eqref{eq:lambdaprime} and \eqref{eq:translationdatum1}.
Let $S'$ be the open orbit in $X(\lambda')$. 
Then, by the description of the image of $f^{\ast}_{\mathcal{T}}$ in \cite[page 95]{ABV}, we deduce
$$
f^{\ast}_{\mathcal{T}}(S,\underline{\mathbb{C}}_{S})\, =\, 
(S',\underline{\mathbb{C}}_{S'}). 
$$
Since $\pi(S',\underline{\mathbb{C}}_{S'})$ is generic, it follows from \eqref{eq:GKdimEQUALITY} and \cite[Theorem 6.2]{Vogan78} that the same holds for 
$\pi(S,\underline{\mathbb{C}}_{S})$, completing the proof of the first part of the theorem.

Recall that, by~\cite[Section~4]{Arthur89},  
to any A-parameter \( \psi \) there corresponds an L-parameter  
\( \varphi_\psi \), and therefore, under~\eqref{eq:KorbitToLmorphismBij},  
a \( {}^{\vee}K \)-orbit \( S_\psi \subset X(\lambda) \).  
Let \( x \in X(\lambda) \) be such that  
\( S_\psi = {}^{\vee}K \cdot x \).  
Then the statement of the theorem concerning the Enhanced Shahidi Conjecture follows from the description in  
\cite[(22.8)(a)–(f)]{ABV} of the A-parameter \( \psi \) of \( G \). 
Indeed, let
\[
h = 
\begin{pmatrix}
1 & 0 \\
0 & -1
\end{pmatrix},
\quad
e =
\begin{pmatrix}
0 & 1 \\
0 & 0
\end{pmatrix},
\quad
f =
\begin{pmatrix}
0 & 0 \\
1 & 0
\end{pmatrix},
\]
denote the standard basis of $\mathfrak{sl}_2(\mathbb{C})$. 
Then, by \cite[(22.8)(e),(f)]{ABV},
we have
$$
d(\psi|_{\mathrm{SL}_2})(e)\in T_{S_\psi,x}^{\ast}X(\lambda).
$$ 
If, in addition, $S_\psi$ is open, then $T^{*}_{S_\psi, x} X(\lambda) = 0$, and hence
$d(\psi|_{\mathrm{SL}_2})(e)=d(\psi|_{\mathrm{SL}_2})(f)=d(\psi|_{\mathrm{SL}_2})(h)=0$, which implies that $\psi|_{\mathrm{SL}_2}$ is trivial. 
Conversely, if \( \psi|_{\mathrm{SL}_2} \) is trivial, then the fact that  
\( S_\psi \) is open follows from~\cite[Proposition~22.9(b)–(c)]{ABV}  
(see, for example,~\cite[Proposition~4.6]{arancibia_characteristic} for further details). 
Hence, the Enhanced Shahidi Conjecture follows directly from the first statement of the theorem.

\end{proof}

As a consequence of the theorem, whenever \( \Pi_S^{\mathrm{mic}}(G(\mathbb{R},\delta_q)) \) contains a generic member, it reduces to the corresponding L-packet \( \Pi_S(G(\mathbb{R},\delta_q)) \).  
As a final remark, we note that, in the course of proving Theorem~\ref{theo:ESC},  
the \( {}^{\vee}W \)-action has been used to provide an alternative proof of the fact that a tempered 
L-packet \( \Pi_{\varphi}(G(\mathbb{R},\delta_q)) \) contains at least one generic member—namely, \( \pi(\xi) \) with \( \xi = (S_\varphi, \underline{\C}_{S_\varphi}) \).  
We believe that the \( {}^{\vee}W \)-action could be further exploited for the purpose of describing micro-packets.


\bibliographystyle{alpha}
\bibliography{reference}
\end{document}